\documentclass[11pt]{article}

\usepackage[nottoc,notlot,notlof]{tocbibind}

\usepackage{amsmath,amssymb,amsfonts,amsthm}
\usepackage{latexsym}
\usepackage{tikz}
\usetikzlibrary{decorations.markings}
\usetikzlibrary{automata,positioning}
\usetikzlibrary{chains,fit,shapes}
\usetikzlibrary{calc}
\usetikzlibrary{arrows}
\usepackage{tkz-graph}
\usetikzlibrary{positioning,calc}
\usetikzlibrary{graphs}
\usetikzlibrary{graphs.standard}
\usetikzlibrary{arrows,decorations.markings}

\newcount\nodecount
\tikzgraphsset{
  declare={subgraph N}%
  {
    [/utils/exec={\global\nodecount=0}]
    \foreach \nodetext in \tikzgraphV
    {  [/utils/exec={\global\advance\nodecount by1}, 
      parse/.expand once={\the\nodecount/\nodetext}] }
  },
  declare={subgraph C}%
  {
    [cycle, /utils/exec={\global\nodecount=0}]
    \foreach \nodetext in \tikzgraphV
    {  [/utils/exec={\global\advance\nodecount by1}, 
      parse/.expand once={\the\nodecount/\nodetext}] }
  }
}

\newtheorem{definition}{Definition}
\newtheorem{theorem}{Theorem}
\newtheorem{lemma}[theorem]{Lemma}

\newtheorem{corollary}[theorem]{Corollary}

\newcommand\p{\circle*{0.3}}

\title{Representing split graphs by words}

\author{Herman~Z.~Q.~Chen\footnote{School of Statistics and Data Science, Nankai University, P.R.\ China. {\bf Email:} zqchern@163.com}, Sergey Kitaev\footnote{Department of Computer and Information Sciences, University of Strathclyde, 26 Richmond Street, Glasgow, G1 1XH, United Kingdom. 
{\bf Email:} sergey.kitaev@cis.strath.ac.uk.}, \  and Akira Saito\footnote{Department of Information Science, Nihon University, Sakurajosui 3-25-40
Setagaya-Ku Tokyo 156--8550, Japan. {\bf Email:} asaito@chs.nihon-u.ac.jp.}}

\begin{document}  

\maketitle

\abstract{There is a long line of research in the literature dedicated to word-representable graphs, which generalize several important classes of graphs. However, not much is known about word-representability of split graphs, another important class of graphs. 

In this paper, we show that threshold graphs, a subclass of split graphs, are word-representable. Further, we prove a number of general theorems on word-representable split graphs, and use them to characterize computationally such graphs with cliques of size 5 in terms of 9 forbidden subgraphs, thus extending the known characterization for word-representable split graphs with cliques of size 4. Moreover, we use split graphs, and also provide an alternative solution, to show that gluing two word-representable graphs in any clique of size at least 2 may, or may not, result in a word-representable graph. The two surprisingly simple solutions provided by us answer a question that was open for about ten years. 
}  

\section{Introduction}\label{sec1}

A graph $G=(V,E)$ is {\em word-representable} iff there exists a word $w$
over the alphabet $V$ such that letters $x$ and $y$, $x\neq y$, alternate in $w$ iff $xy\in E$. Here, by {\em alternation}  of $x$ and $y$ in $w$ we mean that after removing {\em all} letters {\em but} the copies of $x$ and $y$ we either obtain a word $xyxy\cdots$, or a word $yxyx\cdots$. For example, the cycle graph $C_5$ labeled by 1--5 in clock-wise direction can be represented by the word 1521324354. It is easy to see that the class of word-representable graphs is {\em hereditary}. That is, removing a vertex in a word-representable graph results in a word-representable graph.  

Up to date, many papers have been written on the subject~\cite{K17}, and the core of the book \cite{KL15} is devoted to the theory of word-representable graphs. It should also be mentioned that the software produced by Marc Glen \cite{G} is often of great help in dealing with such graphs. Word-representable graphs are important as they generalize several fundamental classes of graphs such as {\em circle graphs}, {\em $3$-colorable graphs} and {\em comparability graphs} \cite{KL15}.

An orientation of a graph is {\em semi-transitive} if it is acyclic, and for any directed path $u_1\rightarrow u_2\rightarrow\cdots\rightarrow u_k$ either there is no edge between $u_1$ and $u_k$, or there is an edge $u_i\rightarrow u_j$ for all $1\leq i<j\leq k$. A key result in the area is the following theorem.  

\begin{theorem}[\cite{HKP16}]\label{key-thm} A graph is word-representable iff it admits a semi-transitive orientation. \end{theorem}

In this paper, we will need the following simple lemma. 

\begin{lemma}[\cite{KLMW17}]\label{lem-tran-orie} Let $K_m$ be a clique in a graph $G$. Then any acyclic orientation of $G$ induces a transitive orientation on $K_m$ with a single source (a vertex with no in-coming edges) and a single sink (a vertex with no out-going edged). \end{lemma}

Even though much is known about word-representable graphs, there is only one paper, namely \cite{KLMW17}, dedicated to the study of the word-representability of {\em split graphs} (considered, e.g.\ in \cite{CH77,FH77,Gol,MP95}), that is, graphs in which the vertices can be partitioned into a clique and an independent set. Section~\ref{split-gr-summary} overviews the most relevant to this paper results in \cite{KLMW17}, that can be summarised as follows: 
\begin{itemize}
\item Spit graphs with cliques of size at most 3 are word-representable.
\item Split graphs in which the clique is of size 4 are characterized by avoiding the four graphs in Figure~\ref{nonRepTri} as induced subgraphs. 
%\item Split graphs  in which vertices in the independent set are of degree at most 2 are characterized by avoiding the graph $T_2$ in Figure~\ref{nonRepTri}  and a certain family of minimal non-word-representable graphs as induced subgraphs (see Theorem~\ref{main-1}). 
\item  Necessary and sufficient conditions for an orientation of a split graph to be semi-transitive are given. 
\end{itemize}

\noindent
The major results in this paper can be summarized as follows:

\begin{itemize}
\item The subclass of split graphs known as {\em threshold graphs} is shown to be word-representable in Theorem~\ref{thm-threshold}. Threshold graphs were first introduced by Chv\'{a}tal and Hammer in \cite{CH77}. A chapter on these graphs appears in \cite{Gol}, and the book \cite{MP95} is devoted to them.
\item Split graphs in which the clique is of size $m$ and clique's vertices are of degree at most $m$ are word-representable (see Theorem~\ref{thm-m+1}).
\item An upper bound  on the number of vertices in the independent set of any given degree in a word-representable split graph is given (see Theorems~\ref{degree-thm1} and~\ref{degree-thm2}).
\item The upper bound is used to characterize computationally split graphs having the clique of size 5 in terms of 9 forbidden subgraphs --- those in Figures~\ref{nonRepTri} and~\ref{nonRepT5-T9} (see Section~\ref{sec-cliques-5-6}).
%\item The upper bound is used to characterize computationally split graphs having the clique of size 6 and vertices of degree at most 3 in the independent set in terms of forbidden subgraphs (see Section~\ref{sec-cliques-5-6}).
\item Word-representability of split-graphs is used in Section~\ref{glueing-sec} to show that gluing  two word-representable graphs in a clique of size at least 2 may result in a non-word-representable graph, which answers a long standing, though unpublished until \cite{K17}, open question. We also give an alternative solution to the problem, which is based on a generalization of a known result (see Section~\ref{constr-2}). 
\end{itemize}

\section{Split graphs and word-representation}\label{split-gr-summary}

Let $S_n=(E_{n-m},K_m)$ be a split graph on $n$ vertices, where the vertices of $S_n$ are 
partitioned into a maximal clique $K_m$ and an independent set $E_{n-m}$  
(the vertices in $E_{n-m}$ are of degree at most $m-1$). 

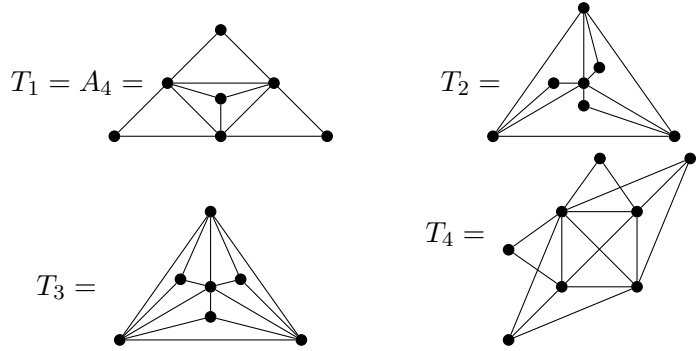
\begin{figure}
\begin{center}

\begin{tabular}{ccccc}
\begin{tikzpicture}[node distance=1cm,auto,main node/.style={fill,circle,draw,inner sep=0pt,minimum size=4pt}]

\node[main node] (1) {};
\node[main node] (2) [below left of=1] {};
\node[main node] (3) [below right of=1] {};
\node[main node] (4) [below left of=2] {};
\node[main node] (5) [below right of=2] {};
\node[main node] (6) [below right of=3] {};
\node[main node] (7) [above of=5, yshift=-0.5cm] {};

\node (8) [left of=2,xshift=-2mm] {$T_1=A_4=$};

\path
(1) edge (4)
(1) edge (6);

\path
(5) edge (2)
(5) edge (3)
(5) edge (4)
(5) edge (6);

\path
(7) edge (2)
(7) edge (3)
(7) edge (5);

\path
(2) edge (3);
\end{tikzpicture}

& 

\ \ \ 

&

\begin{tikzpicture}[node distance=1cm,auto,main node/.style={fill,circle,draw,inner sep=0pt,minimum size=4pt}]

\node[main node] (1) {};
\node[main node] (2) [below of=1] {};
\node[main node] (3) [below left of=2, xshift=-0.5cm] {};
\node[main node] (4) [below right of=2, xshift=0.5cm] {};
\node[main node] (5) [left of=2, xshift=0.6cm] {};
\node[main node] (6) [below of=2, yshift=0.7cm] {};
\node[main node] (7) [above right of=2, xshift=-0.5cm, yshift=-0.5cm] {};
\node (8) [left of=2,  xshift=-0.5cm] {$T_2=$};

\path
(1) edge (2)
(1) edge (4)
(1) edge (7)
(1) edge (3);

\path
(2) edge (3)
(2) edge (4)
(2) edge (5)
(2) edge (6)
(2) edge (7);

\path
(3) edge (4)
(3) edge (5);

\path
(6) edge (4);

\end{tikzpicture}

\\

\begin{tikzpicture}[node distance=1cm,auto,main node/.style={fill,circle,draw,inner sep=0pt,minimum size=4pt}]

\node[main node] (1) {};
\node[main node] (2) [below of=1] {};
\node[main node] (3) [below left of=2,xshift=-0.5cm] {};
\node[main node] (4) [below right of=2,xshift=0.5cm] {};
\node[main node] (5) [below of=2,yshift=0.6cm] {};
\node[main node] (6) [left of=2,yshift=0.1cm,xshift=0.6cm] {};
\node[main node] (7) [right of=2,yshift=0.1cm,xshift=-0.6cm] {};

\node [above left of=3] {$T_3=$};

\path
(2) edge (6)
(2) edge (7)
(5) edge (3)
(5) edge (4)
(5) edge (2)
(1) edge (3)
(1) edge (4)
(6) edge (3)
(7) edge (4)
(1) edge (6)
(1) edge (7)
(2) edge (3)
(2) edge (4)
(4) edge (3)
(1) edge (2);
\end{tikzpicture}

& 

\ \ \ 

& 

\begin{tikzpicture}[node distance=1cm,auto,main node/.style={fill,circle,draw,inner sep=0pt,minimum size=4pt}]

\node[main node] (1) {};
\node[main node] (2) [right of=1] {};
\node[main node] (3) [below of=2] {};
\node[main node] (4) [left of=3] {};
\node[main node] (5) [above right of=1,xshift=-0.2cm] {};
\node[main node] (6) [above right of=2] {};
\node[main node] (7) [below left of=1,yshift=0.2cm] {};
\node[main node] (8) [below left of=4] {};

\node [above left of=7,yshift=-0.5cm] {$T_4=$};

\path
(1) edge (2)
(1) edge (3)
(1) edge (4)
(1) edge (5)
(1) edge (6)
(1) edge (7)
(1) edge (8);

\path
(2) edge (3)
(2) edge (4)
(2) edge (5)
(2) edge (6);

\path
(3) edge (4)
(3) edge (6)
(3) edge (8);

\path
(4) edge (8)
(4) edge (7);

\end{tikzpicture}
\end{tabular}

\caption{\label{nonRepTri} The minimal non-word-representable split graphs $T_1$, $T_2$, $T_3$, $T_4$}
\end{center}
\end{figure}

In this section, we overview most relevant to us results in \cite{KLMW17}.
%\begin{theorem}[\cite{KLMW17}] Let $S_n=(E_{n-m},K_m)$ be a split graph and $m\leq 3$. Then $S_n$ is word-representable. \end{theorem}

\begin{lemma}[\cite{KLMW17}]\label{lemma-assumptions} Let $S_n=(E_{n-m},K_m)$ be a split graph, and a spit graph $S_{n+1}$ is obtained from $S_n$ by either adding a vertex of degree $0$ or  $1$, or by ``copying'' a vertex, that is, by adding a vertex whose neighbourhood is identical to the neighbourhood of a vertex in $S_n$ (if copying a vertex in $K_m$, then the copy is connected to the original vertex). Then $S_n$ is word-representable iff $S_{n+1}$ is word-representable. \end{lemma}

\begin{definition}\label{def-K-triang} For $\ell\geq 3$, the graph $K^{\triangle}_\ell$ is obtained from the complete graph $K_\ell$ labeled by $1,\ldots, \ell$, by adding a vertex $i'$ of degree $2$ connected to vertices $i$ and $i+1$ for each $i\in \{1,\ldots,\ell-1\}$. Also, a vertex $\ell'$ connected to the vertices $1$ and $\ell$ is added.\end{definition}

\begin{theorem}[\cite{KLMW17}]\label{thm-K-Tri-m-w-r} $K^{\triangle}_\ell$ is word-representable. \end{theorem}

\begin{definition}\label{def-Aell} For $\ell\geq 4$, let $A_\ell$ be the graph obtained from $K^{\triangle}_{\ell-1}$ by adding a vertex $\ell$ connected to the vertices $1,\ldots,\ell-1$ and no other vertices.  Note that $A_4=T_1$ in Figure~\ref{nonRepTri}.  \end{definition}

\begin{theorem}[\cite{KLMW17}]\label{Aell-min-non-repres} $A_\ell$ is a minimal non-word-representable graph. \end{theorem}

%\begin{theorem}[\cite{KLMW17}]\label{main-1} Let $m\geq 1$ and $S_n=(E_{n-m},K_m)$ be a split graph. Also, let the degree of any vertex in $E_{n-m}$ be at most $2$. Then $S_n$ is word-representable iff $S_n$ does not contain the graphs $T_2$ in Figure~\ref{nonRepTri} and $A_{\ell}$ in Definition~\ref{def-Aell} as induced subgraphs. \end{theorem}
%
\begin{theorem}[\cite{KLMW17}]\label{main-2} Let $S_n=(E_{n-4},K_4)$ be a split graph. Then  $S_n$ is word-representable iff $S_n$ does not contain the graphs $T_1$, $T_2$, $T_3$ and $T_4$ in Figure~\ref{nonRepTri} as induced subgraphs.\end{theorem}

Let $S_n=(E_{n-m},K_m)$ be a word-representable split graph. Then, by Theorem~\ref{key-thm}, $S_n$ admits a semi-transitive orientation. Further, by Lemma~\ref{lem-tran-orie} we known that any such orientation induces a transitive orientation on $K_m$ with the longest directed path $\vec{P}$. Theorems~\ref{semi-tran-groups} and~\ref{relative-order} below describe the structure of semi-transitive orientations in an arbitrary word-representable split graph. 

\begin{theorem}[\cite{KLMW17}]\label{semi-tran-groups} Any semi-transitive orientation of $S_n=(E_{n-m},K_m)$ subdivides the set of all vertices in $E_{n-m}$ into three, possibly empty, groups corresponding to each of the following types, where $\vec{P}= p_1\rightarrow\cdots\rightarrow p_m$ is the longest directed path in $K_m$: 
\begin{itemize}
\item A vertex in $E_{n-m}$ is of {\em type A} if it is a source and is connected to all vertices in $\{p_i,p_{i+1},\ldots, p_j\}$ for some $1\leq i\leq j\leq m$;
\item A vertex in $E_{n-m}$ is of {\em type B} if it is a sink and is connected to all vertices in $\{p_i,p_{i+1},\ldots, p_j\}$ for some $1\leq i\leq j\leq m$; 
\item A vertex $v\in E_{n-m}$ is of {\em type C} if there is an edge $x\rightarrow v$ for each $x\in I_v=\{p_1,p_2,\ldots, p_i\}$ and there is an edge $v\rightarrow y$ for each $y\in O_v=\{p_j,p_{j+1},\ldots, p_m\}$ for some $1\leq i< j\leq m$.
\end{itemize} 
\end{theorem}

There are additional restrictions, given by the next theorem, on relative positions of the neighbours of vertices of types A, B and C. 

\begin{theorem}[\cite{KLMW17}]\label{relative-order} Let $S_n=(E_{n-m},K_m)$ be oriented semi-transitively with $\vec{P}= p_1\rightarrow\cdots\rightarrow p_m$. For a vertex $x\in E_{n-m}$ of type C,
there is no vertex $y\in E_{n-m}$ of type A or B, which is connected to both $p_{|I_x|}$ and $p_{m-|O_x|+1}$. Also, there is no vertex $y\in E_{n-m}$ of type C such that either $I_y$, or $O_y$ contains both $p_{|I_x|}$ and $p_{m-|O_x|+1}$.
\end{theorem}

One can now classify semi-transitive orientations on split graphs.

\begin{theorem}[\cite{KLMW17}]\label{main-orientation} An orientation of a split graph $S_n=(E_{n-m},K_m)$ is semi-transitive iff 
\begin{itemize} 
\item $K_m$ is oriented transitively,
\item each vertex in $E_{n-m}$ is of one of the three types in Theorem~\ref{semi-tran-groups}, 
\item the restrictions in Theorem~\ref{relative-order} are satisfied. 
\end{itemize}\end{theorem}

The following corollary of Theorem~\ref{main-orientation} generalizes Theorem~\ref{thm-K-Tri-m-w-r} 
%(which is the case $k=2$ in the corollary)
. 

\begin{corollary}[\cite{KLMW17}]\label{K-ell-k} Let the split graph $K_{\ell}^k$ be obtained from the complete graph $K_{\ell}$, whose vertices are drawn on a circle, by adding $\ell$ vertices so that 
\begin{itemize}
\item each such vertex is connected to $k$ consecutive (on the circle) vertices in $K_{\ell}$; 
\item neighbourhoods of all these vertices are distinct; and
\item $\ell\geq 2k-1$.
\end{itemize}
Then $K_{\ell}^k$ is word-representable. \end{corollary}

The following theorem allows us to treat vertices of types A or B in the same way and to refer to them as vertices of type A\&B. 

\begin{theorem}[\cite{KLMW17}]\label{intercahnging} 
Let $S_n=(E_{n-m},K_m)$ be semi-transitively oriented. Then, any vertex in $E_{n-m}$ of type A can be replaced by a vertex of type B, and vice versa, keeping orientation semi-transitive. 
\end{theorem}

\section{Threshold graphs and split graphs with restricted vertex degree in the clique}\label{threshold-sec}

A {\em threshold graph} is a graph that can be constructed from the one-vertex graph by repeated applications of the following two operations:
\begin{itemize}
\item[(1)] Addition of a single isolated vertex to the graph.
\item[(2)] Addition of a single dominating vertex to the graph, i.e.\ a single vertex that is connected to all other vertices.
\end{itemize}
It is not difficult to see that any threshold graph is a split graph. 

\begin{theorem}\label{thm-threshold} Any threshold graph $S_n$ is word-representable. \end{theorem}

\begin{proof} Label the vertices in the order they were added to $S_n$: $1, 2,\ldots,n$.  Note that no matter which operation is applied, the vertices $1$ and $2$ will have the same neighbourhood modulo them possibly being connected to each other. Thus, by Lemma~\ref{lemma-assumptions}, removing vertex $1$ does not affect word-representability of the graph. But then, the vertices $2$ and $3$ will have the same neighbourhood modulo them possibly being connected to each other.  Thus, by Lemma~\ref{lemma-assumptions}, removing vertex $2$ does not affect word-representability of the graph. Continuing in the same way, we see that $S_n$ is word-representable iff the one-vertex graph (labeled by $n$) is word-representable, which is trivially the case.  \end{proof}

\begin{theorem}\label{thm-m+1} Let $S_n=(E_{n-m},K_m)$ be a split graph such that each vertex $v$ in $K_m$ is of degree at most $m$, i.e.\ the degree of $v$ is $m-1$ or $m$. Then $S_n$ is word-representable.\end{theorem}

\begin{proof} Orient $K_m$ in an arbitrary transitive way, which will result, by Lemma~\ref{lem-tran-orie}, in a longest directed path $\vec{P}=p_1\rightarrow\cdots\rightarrow p_m$. Because each vertex in $K_m$ can be connected to at most one vertex in $E_{n-m}$, we can clearly permute the vertices in $\vec{P}$ (resulting in a different transitive orientation of $K_m$) so that the neighbourhood of each vertex in $E_{n-m}$ consists of a number of consecutive vertices in  $\vec{P}$, and these neighbourhoods do not overlap. Making each vertex in $E_{n-m}$ either of type A, or of type B, we can apply Theorem~\ref{main-orientation} to see that $S_n$ is semi-transitively oriented, and thus, by Theorem~\ref{key-thm}, $S_n$ is word-representable. \end{proof}

\begin{theorem}\label{large-ext-degree} Let $S_n=(E_{n-m},K_m)$ be a split graph, where the neighbourhoods of all vertices in $E_{n-m}$ are distinct. If $K_m$ has a vertex $v$ connected to at least  $d+1$ vertices of degree $d\leq m-2$ in $E_{n-m}$, then $S_n$ is not word-representable.  \end{theorem}

\begin{proof} Supposed $S_n$ is word-representable, so that $S_n$ can be oriented semi-transitively by Theorem~\ref{key-thm}. Let $v_1,v_2,\ldots,v_{d+1}\in E_{n-m}$ be vertices of degree $d$ connected to $v$. By Theorem~\ref{main-orientation}, the distinct neighbourhoods of $v_i$, $1\leq i\leq d+1$, form consecutive cyclic intervals of $d$ vertices on the directed path $\vec{P}$, each of which contains $v$. Contradiction with the fact that $v$ can be covered by at most $d$ distinct intervals of $d$ vertices. \end{proof}

\section{Properties of degrees in the independent sets in word-representable split graphs}\label{degree-sec}

An immediate corollary of Lemma~\ref{lemma-assumptions} is that in our studies of word-representable split graphs $S_n=(E_{n-m},K_m)$ we can assume that 
\begin{itemize} 
\item no two vertices in $S_n$ have the same set of neighbours modulo vertices being connected to each other, so that
\item at most one vertex in $K_m$ is not connected to any vertex in $E_{n-m}$, and
 each vertex in $E_{n-m}$ is of degree at least 2.
\end{itemize}

However, when studying minimal non-word-representable subgraphs of a split graph, we can make other assumptions as well, which allow a reduction of the space of possible solutions, e.g.\ when  proceeding with a computer-aided search. The following two theorems are very useful.  

\begin{theorem}\label{degree-thm1} Let $S_n=(E_{n-m},K_m)$ be a word-representable graph, $m\geq 3$, and $2\leq d\leq \frac{m+1}{2}$. Then, $E_{n-m}$ contains at most $m$ vertices of degree $d$ whose neighbourhoods are distinct. This bound is achievable. \end{theorem}

\begin{proof} By Theorem~\ref{key-thm}, $S_n$ admits a semi-transitive orientation, in which the neighbourhoods of the vertices in $E_{n-m}$, by Theorem~\ref{main-orientation}, are consecutive on the directed path $\vec{P}$ when read cyclicly. There are $m$ distinct consecutive (cyclic) intervals of length $d$, which gives the upper bound. Finally, since $d\leq \frac{m+1}{2}$, the restrictions in Theorem~\ref{relative-order} are satisfied, which makes the bound achievable (letting $\ell=m$ and $k=d$ in Corollary~\ref{K-ell-k}, we obtain the graph achieving the bound).\end{proof}

\begin{theorem}\label{degree-thm2} Let $S_n=(E_{n-m},K_m)$ be a word-representable graph, $m\geq 4$, and $\frac{m+1}{2} < d\leq m-1$. Then, $E_{n-m}$ contains at most $m-d+1$ vertices of degree $d$ whose neighbourhoods are distinct. This bound is achievable. \end{theorem}

\begin{proof} By Theorem~\ref{key-thm}, $S_n$ admits a semi-transitive orientation. Note that $m-d+1$ is the number of distinct non-cyclic consecutive intervals of vertices on the path $\vec{P}$. Any number of these intervals can be the neighbourhoods of type A\&B vertices by Theorem~\ref{main-orientation}, so there exists the split graph $S_n$ with the maximum number of type A\&B vertices showing that the bound is achievable.  

\begin{figure}
\begin{center}
 \begin{tikzpicture}[scale=1, transform shape]
 \tikzstyle mycircle=[draw,shape=circle,minimum 
width=2.5cm,minimum height=1.5cm] 
  \tikzstyle{vertex}=[circle,fill=black,minimum size=4pt,inner sep=-1pt]
 \node[mycircle] (A) {$K_m$};
\node[mycircle] (B) [right = of A] {$K_m$};
\node[mycircle] (C) [right = of B] {$K_m$};
\path (A.0) node[vertex] (A1){} node[right]{$c$}
        (A.180) node[vertex](A2){} node[left]{$a$}
        (A.270)node[vertex](A3){}node[below]{$b$};
\draw (A1)--(A3)--(A2);

\path (B.70) node[vertex] (B1){} node[right,yshift=1mm]{$d$}
        (B.160) node[vertex](B2){} node[left]{$c$}
        (B.180)node[vertex](B3){}node[left]{$a$}
        (B.270) node[vertex](B4){} node[below]{$b$};
\draw (B1)--(B2);
\draw (B3)--(B4);

\path (C.220) node[vertex] (C1){} node[left]{$a$}
        (C.320) node[vertex](C2){} node[right]{$b$}
        (C.245)node[vertex](C3){}
        (C.270)node[vertex] {} 
         (C.295)node[vertex](C4){} ;
\draw (C1)--(C2);

 \end{tikzpicture}

\caption{\label{intervals-fig} A schematic representation of consecutive cyclic intervals of vertices on $\vec{P}$ using chords to support the proof of Theorem~\ref{degree-thm2}. }
\end{center}
\end{figure}
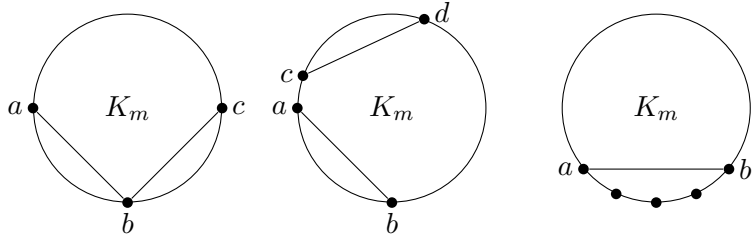

Next we prove that the bound can never be exceeded.  To do this, we use the schematic way to represent consecutive cyclic intervals of vertices on $\vec{P}$ given in Figure~\ref{intervals-fig}. In that figure, the vertices in $K_m$ are placed on a circle in clockwise direction in the order they appear in the directed path $\vec{P}$, and the chord $ab$ represents the (cyclic) interval of vertices of length $d$ that starts at $a$ and ends at $b$. If such an interval corresponds to the neighbourhood of a vertex $v$ in $E_{n-m}$, then $v$ is of type A\&B if $a$ is before $b$ in $\vec{P}$, and $v$ is of type C if $b$ is before $a$ in $\vec{P}$.

Our first observation is that no matter what the semi-transitive orientation of $S$ is, no two chords corresponding to the neighbourhoods of vertices in $E_{n-m}$ can share an endpoint. Indeed, suppose $ab$ and $bc$ are chords as in the leftmost picture in Figure~\ref{intervals-fig}. But then, because $d>\frac{m+1}{2}$, at least one of the cords $ab$ and $bc$ corresponds to the neighbourhood of a vertex in $E_{n-m}$ of type C. Suppose $ab$ corresponds to a vertex of type C (the second case is analogous). But then, the interval given by $bc$ covers both of $a$ and $b$, which contradicts to Theorem~\ref{relative-order}.

Our second observation is that no matter what the semi-transitive orientation of $S$ is, any two chords corresponding to the neighbourhoods of vertices in $E_{n-m}$ must intersect each other, that is, the situation presented in the second picture in Figure~\ref{intervals-fig} is not possible. Indeed, if $ab$ and $dc$ do not intersect each other, then at least one of them corresponds to a vertex in $E_{n-m}$ of type C because $d>\frac{m+1}{2}$. But then, we obtain exactly the same contradiction with Theorem~\ref{relative-order} as in the first observation.

Finally, suppose $ab$ represents the neighbourhood of a vertex in $E_{n-m}$ as in the rightmost picture in Figure~\ref{relative-order}. The chords representing any other neighbourhoods must have (exactly one) of their endpoints  among the indicated $m-d$ vertices in that picture by the second observation. However, by the first observation, each of the $m-d$ vertices can be connected to at most one chord, which results in the maximum possible total amount of chords, and thus vertices in $E_{n-m}$ of degree $d>\frac{m+1}{2}$, be $m-d+1$, as desired.
\end{proof}

\section{Characterizing word-representable split graphs with cliques of size 5}\label{sec-cliques-5-6}

Applying Theorems~\ref{degree-thm1} and~\ref{degree-thm2} we see that in a word-representable graph $S_n=(E_{n-5},K_5)$ we can have at most two vertices of degree 4, at most  five vertices of degree 3, and at most five vertices of degree 2 (recall that vertices of degree 1 never affect word-representability). 
%Also, in a word-representable graph $S_n=(E_{n-6},K_6)$ we can have at most two vertices of degree 5, at most three vertices of degree 4, at most  six vertices of degree 3, and at most six vertices of degree 2.
 
\begin{figure}
\begin{center}

\begin{tabular}{ccc}
\begin{tikzpicture}[scale=0.8, transform shape]
  \tikzstyle{vertex}=[circle,fill=black,minimum size=4pt,inner sep=-1pt]
  \foreach \vx in {1,...,9}
    \node[vertex,xshift=6cm,yshift=.5cm] (v-\vx) at (\vx*-40:2cm) {};
  \foreach \x in {1,...,5} {
      \foreach \y in {\x,...,5} {
        \draw (v-\x) -- (v-\y);}
      };
 \foreach \from/\to in {1/6,1/9,2/6,2/7,3/7,3/8,4/8,4/9}
    { \draw (v-\from) -- (v-\to);}
\node[left = of v-5, xshift=0.8cm] {$T_5=A_4=$};
 \end{tikzpicture}

&
\begin{tikzpicture}[scale=0.8, transform shape]
  \tikzstyle{vertex}=[circle,fill=black,minimum size=4pt,inner sep=-1pt]
  \foreach \vx in {1,...,9}
    \node[vertex,xshift=6cm,yshift=.5cm] (v-\vx) at (\vx*-40:2cm) {};
  \foreach \x in {1,...,5} {
      \foreach \y in {\x,...,5} {
        \draw (v-\x) -- (v-\y);}
      };
 \foreach \from/\to in {1/6,1/7,1/8,2/6,2/7,2/9,3/6,4/8,5/9}
    { \draw (v-\from) -- (v-\to);}
\node[left = of v-5, xshift=0.6cm] {$T_6=$};
 \end{tikzpicture}
 \\
\begin{tikzpicture}[scale=0.8, transform shape]
  \tikzstyle{vertex}=[circle,fill=black,minimum size=4pt,inner sep=-1pt]
  \foreach \vx in {1,...,9}
    \node[vertex,xshift=0cm,yshift=0cm] (v-\vx) at (\vx*-40:2cm) {};
  \foreach \x in {1,...,5} {
      \foreach \y in {\x,...,5} {
        \draw (v-\x) -- (v-\y);}
      };
 \foreach \from/\to in {1/2, 1/3, 1/4, 1/5, 1/6, 1/8, 2/3, 
 2/4, 2/5, 2/7, 2/8, 3/4, 3/5, 3/6, 
 3/9, 4/5, 4/7, 4/9, 5/8, 5/9}
    { \draw (v-\from) -- (v-\to);}
\node[left] at (-2cm,0)  {$T_7=$};
 \end{tikzpicture}
  &
\begin{tikzpicture}[scale=0.8, transform shape]  
\tikzstyle{vertex}=[circle,fill=black,minimum size=4pt,inner sep=-1pt]
  \foreach \vx in {1,...,9}
    \node[vertex,xshift=0cm,yshift=0cm] (v-\vx) at (\vx*-40:2cm) {};
  \foreach \x in {1,...,5} {
      \foreach \y in {\x,...,5} {
        \draw (v-\x) -- (v-\y);}
      };
 \foreach \from/\to in {1/2, 1/3, 1/4, 1/5, 1/6, 1/8, 1/9, 
 2/3, 2/4, 2/5, 2/7, 2/8, 2/9, 3/4, 
 3/5, 3/6, 3/9, 4/5, 4/7, 4/9, 5/8}
    { \draw (v-\from) -- (v-\to);}
\node[left] at (-2cm,0)   {$T_8=$};
 \end{tikzpicture}

\end{tabular}

\begin{tikzpicture}[scale=0.8, transform shape]
  \tikzstyle{vertex}=[circle,fill=black,minimum size=4pt,inner sep=-1pt]
  \foreach \vx in {1,...,10}
    \node[vertex,xshift=0cm,yshift=0cm] (v-\vx) at (\vx*-36:2cm) {};
  \foreach \x in {1,...,5} {
      \foreach \y in {\x,...,5} {
        \draw (v-\x) -- (v-\y);}
      };
 \foreach \from/\to in{1/2, 1/3, 1/4, 1/5, 1/6, 1/7, 1/9, 
 1/10, 2/3, 2/4, 2/5, 2/8, 2/9, 3/4, 
 3/5, 3/6, 3/8, 3/10, 3/9, 4/5, 4/7, 
 4/10, 5/9, 5/10}
    { \draw (v-\from) -- (v-\to);}
    \node[left] at (-2.2cm,0)  {$T_9=$};
 \end{tikzpicture}

\caption{\label{nonRepT5-T9} The minimal non-word-representable split graphs $T_5$--$T_9$}
\end{center}
\end{figure}
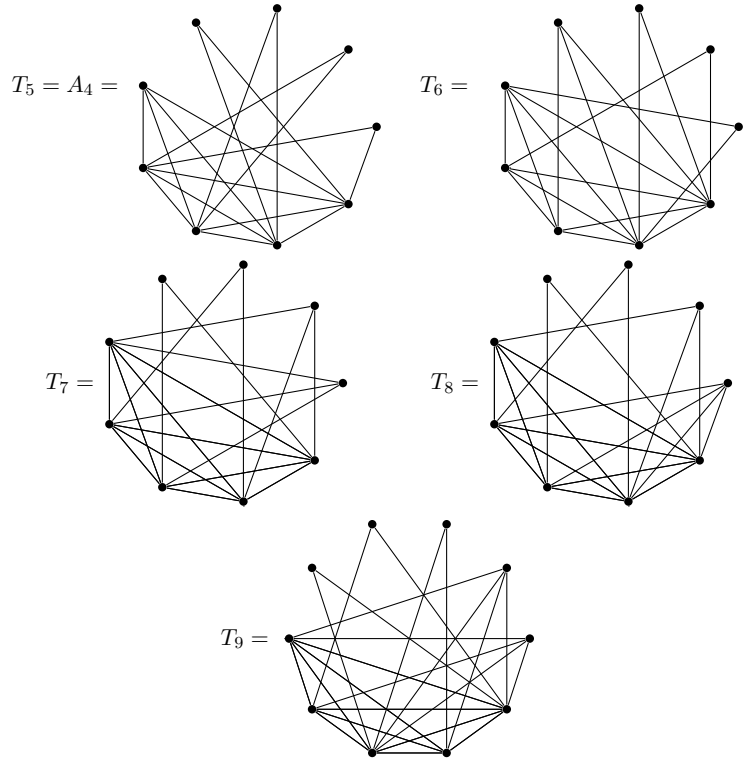

Clearly, the minimal non-word-representable graphs in Figure~\ref{nonRepTri} must be avoided when considering   $K_5$. Computational experiments for $S_n=(E_{n-5},K_5)$, which were possible due to the assumptions discussed above, reveal 5 more minimal non-word-representable graphs presented in Figure~\ref{nonRepT5-T9}. One of these graphs is $A_4$ (see Definition~\ref{def-Aell}) whose minimality and non-word-representability is given by Theorem~\ref{Aell-min-non-repres}. We conclude the section with proving that the graphs $T_6$--$T_9$ in Figure~\ref{nonRepT5-T9} are minimal non-word-representable graph.

\begin{theorem} The graph $T_6$ in Figure~\ref{nonRepT5-T9} is a minimal non-word-representable graph.\end{theorem}

\begin{figure}
\begin{center}
\begin{tikzpicture}[scale=1, transform shape]
 \tikzstyle{vertex}=[circle,fill=black,minimum size=3pt,inner sep=-1pt]
\path node[draw,shape=circle,minimum width=2cm,minimum height=2cm] (A) {$K_5$};
\foreach \vx/\lab in {45/j,135/i,210/x,270/y,330/z}
  \node[vertex] (v-\lab) at (A.\vx) {};
\foreach \vx in {1,...,4}
  \node[vertex]  (v-\vx) at (\vx*90:1.5cm) {} ;
 \foreach \from/\to in {2/i,2/x,1/i,1/j,4/j,4/z,3/i,3/j,3/y}
    { \draw (v-\from) -- (v-\to);}
\path (v-1) node[left] {{\small $b$}}
    (v-2) node[left] {{\small $a$}}
    (v-3) node[left] {{\small $d$}}
    (v-4) node[right] {{\small $c$}}
    (v-i) node[left] {{\small $i$}}
    (v-j) node[right] {{\small $j$}}
    (v-x) node[left] {{\small $x$}}
    (v-y) node[above] {{\small $y$}}
    (v-z) node[right] {{\small $z$}};
 \end{tikzpicture}
\caption{\label{T6-fig} Proving non-word-representability of $T_6$}
\end{center}
\end{figure}
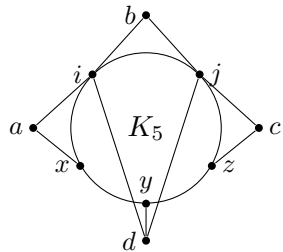

\begin{proof} We begin with proving non-word-representability of $T_6$. Suppose $T_6$ is word-representable, and thus, by Theorem~\ref{key-thm}, it can be oriented semi-transitively.  Pick any such semi-transitive orientation of $T_6$. Then, by Theorem~\ref{main-orientation}, the neighbourhood of a vertex of degree 2 must  be two vertices staying next to each other, possibly cyclicly (if they are the source and the sink), on the path $\vec{P}$, as shown in Figure~\ref{T6-fig} (where the five vertices of $\vec{P}$ are placed on a circle; note that in our argument it is not important where the source and sink are). But then, the neighbourhood of the vertex $d$ of degree 3 is forced to be non-consecutive vertices on $\vec{P}$. Contradiction with Theorem~\ref{main-orientation}.

For proving the minimality of $T_6$, we consider removing each of the vertices in $T_6$ (one at a time) and, if necessary, describe a permutation of vertices of $K_5$, which results in all neighbourhoods of vertices in $E_{n-5}$ be consecutive intervals on $\vec{P}$, or on whatever remains from $\vec{P}$ (which is still transitively oriented); then, Theorem~\ref{main-orientation} can be used to obtain a semi-transitive orientation of the resulting graph proving its word-representability by Theorem~\ref{key-thm}. 

\begin{itemize}
\item The vertices $a$ and $c$ are clearly symmetric, so we can consider removing $a$ and skip considering removing $c$. In the case of $a$ removed, swap $x$ and $y$  to obtain the desired result. 
\item If $d$ is removed, all intervals become consecutive.
\item If $b$ is removed, place $y$ between $i$ and $j$  to obtain the result.
\item The vertices $x$ and $z$ are clearly symmetric, so we can consider removing $x$ and skip considering removing $y$. The vertex $a$ becomes of degree 1 and can be also removed by Lemma~\ref{lemma-assumptions}. All intervals become consecutive.
\item If the vertex $y$ is removed, then the vertices $d$ and $b$ have the same neighbourhoods, and one of them can be removed by Lemma~\ref{lemma-assumptions}. All intervals become consecutive.
\item The vertices $i$ and $j$ are clearly symmetric, so we can consider removing $i$ and skip considering removing $j$. If $i$ is removed, $a$ and $b$ become of degree 1 and can be removed by Lemma~\ref{lemma-assumptions}. Swap $x$ and $y$.
\end{itemize}
Our proof is completed.
\end{proof}

\begin{theorem} The graph $T_7$ in Figure~\ref{nonRepT5-T9} is a minimal non-word-representable graph.\end{theorem}

\begin{figure}
\begin{center}
\begin{tikzpicture}[scale=1, transform shape]
 \tikzstyle{vertex}=[circle,fill=black,minimum size=3pt,inner sep=-1pt]
\path node[draw,shape=circle,minimum width=2cm,minimum height=2cm] (A) {$K_5$};
\foreach \vx/\lab in {45/j,135/i,210/x,270/y,330/z}
  \node[vertex] (v-\lab) at (A.\vx) {};
\foreach \vx in {1,...,4}
  \node[vertex]  (v-\vx) at (\vx*90:1.5cm) {} ;
 \foreach \from/\to in {2/i,2/x,1/i,1/j,1/y,4/j,4/z,3/x,3/z,3/y}
    { \draw (v-\from) -- (v-\to);}
\path (v-1) node[left] {{\small $b$}}
    (v-2) node[left] {{\small $a$}}
    (v-3) node[left] {{\small $d$}}
    (v-4) node[right] {{\small $c$}}
    (v-i) node[left] {{\small $i$}}
    (v-j) node[right] {{\small $j$}}
    (v-x) node[left] {{\small $x$}}
    (v-y) node[above,xshift=-1.5mm] {{\small $y$}}
    (v-z) node[right] {{\small $z$}};
 \end{tikzpicture}
\caption{\label{T7-fig} Proving non-word-representability of $T_7$}
\end{center}
\end{figure}
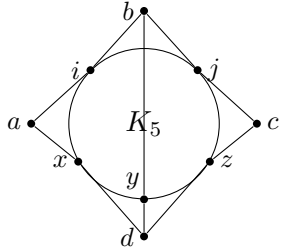

\begin{proof} We begin with proving non-word-representability of $T_7$. Suppose $T_7$ is word-representable, and thus, by Theorem~\ref{key-thm}, it can be oriented semi-transitively.  Pick any such semi-transitive orientation of $T_7$. Then, by Theorem~\ref{main-orientation}, the neighbourhoods of vertices of degree 2 must  be consecutive, so since they are also disjoint, without loss of generality the degree 2 vertices are positioned as in Figure~\ref{T7-fig} (where the five vertices of $\vec{P}$ are placed on a circle; note that in our argument it is not important where the source and sink are). But then, since the vertices of degree 3 have symmetric properties (their neighbourhoods contain one vertex from each of vertices $a$ and $c$ neighbourhoods and vertex $y$) we see that there is no way for both neighbourhoods of $b$ and $d$ to be consecutive on $\vec{P}$. Contradiction with Theorem~\ref{main-orientation}.

For proving the minimality of $T_7$, we consider removing each of the vertices in $T_7$ (one at a time) and, if necessary, describe a permutation of vertices of $K_5$, which results in all neighbourhoods of vertices in $E_{n-5}$ be consecutive intervals on $\vec{P}$, or on whatever remains from $\vec{P}$ (which is still transitively oriented); then, Theorem~\ref{main-orientation} can be used to obtain a semi-transitive orientation of the resulting graph proving its word-representability by Theorem~\ref{key-thm}. 

\begin{itemize}
\item The vertices $a$ and $c$ are clearly symmetric, so we can consider removing $a$ and skip considering removing $c$. In the case of $a$ removed, swap $x$ and $y$ to obtain the desired result. 
\item If $b$ is removed, all intervals become consecutive. If $d$ is removed, then place $y$ between $i$ and $j$  to obtain the desired result. 
\item The vertices $i$ and $j$ are clearly symmetric, so we can consider removing $i$ and skip considering removing $j$. The vertex $a$ becomes of degree 1 and can be also removed by Lemma~\ref{lemma-assumptions}. 
%Swap $x$ and $y$. 
The obtained graph is word-representable by Theorem~\ref{main-2}.
\item The vertices $x$ and $z$ are clearly symmetric, so we can consider removing $x$ and skip considering removing $z$. The vertex $a$ becomes of degree 1 and can be also removed by Lemma~\ref{lemma-assumptions}. 
%All intervals become consecutive.
The obtained graph is word-representable by Theorem~\ref{main-2}.
\item If $y$ is removed, then the graph is isomorphic to $K^{\triangle}_4$ and it is word-representable by Theorem~\ref{thm-K-Tri-m-w-r}. 
\end{itemize}
Our proof is completed.
 \end{proof}

\begin{theorem} The graph $T_8$ in Figure~\ref{nonRepT5-T9} is a minimal non-word-representable graph.\end{theorem}

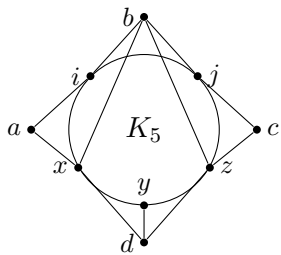
\begin{figure}
\begin{center}
\begin{tikzpicture}[scale=1, transform shape]
 \tikzstyle{vertex}=[circle,fill=black,minimum size=3pt,inner sep=-1pt]
\path node[draw,shape=circle,minimum width=2cm,minimum height=2cm] (A) {$K_5$};
\foreach \vx/\lab in {45/j,135/i,210/x,270/y,330/z}
  \node[vertex] (v-\lab) at (A.\vx) {};
\foreach \vx in {1,...,4}
  \node[vertex]  (v-\vx) at (\vx*90:1.5cm) {} ;
 \foreach \from/\to in {2/i,2/x,1/i,1/j,1/x,1/z,4/j,4/z,3/x,3/z,3/y}
    { \draw (v-\from) -- (v-\to);}
\path (v-1) node[left] {{\small $b$}}
    (v-2) node[left] {{\small $a$}}
    (v-3) node[left] {{\small $d$}}
    (v-4) node[right] {{\small $c$}}
    (v-i) node[left] {{\small $i$}}
    (v-j) node[right] {{\small $j$}}
    (v-x) node[left] {{\small $x$}}
    (v-y) node[above] {{\small $y$}}
    (v-z) node[right] {{\small $z$}};
 \end{tikzpicture}
\caption{\label{T8-fig} Proving non-word-representability of $T_8$ }
\end{center}
\end{figure}

\begin{proof} We begin with proving non-word-representability of $T_8$. Suppose $T_8$ is word-representable, and thus, by Theorem~\ref{key-thm}, it can be oriented semi-transitively.  Pick any such semi-transitive orientation of $T_8$. Then, by Theorem~\ref{main-orientation}, the neighbourhoods of all vertices in the independent set must be consecutive intervals, and the only way to arrange this is shown in Figure~\ref{T8-fig} (where the five vertices of $\vec{P}$ are placed on a circle and the orientation of the longest path is assumed to be in clockwise direction). But then we obtain a contradiction with Theorem~\ref{relative-order}. Indeed, if  $b$ is of type $C$ then $d$ must be of type A\&B, but $x$ and $z$ are in the neighbourhood of $d$. On the other hand, if $d$ is of type $C$ then $b$ must be of type A\&B, but $x$ and $z$ are in the neighbourhood of $b$. Thus, $T_8$ is not word-representable. 

For proving the minimality of $T_8$, we consider removing each of the vertices in $T_8$ (one at a time) and, if necessary, describe a permutation of vertices of $K_5$, which results in all neighbourhoods of vertices in $E_{n-5}$ be consecutive intervals on $\vec{P}$, or on whatever remains from $\vec{P}$ (which is still transitively oriented); then, Theorem~\ref{main-orientation} can be used to obtain a semi-transitive orientation of the resulting graph proving its word-representability by Theorem~\ref{key-thm}. 

\begin{itemize}
\item The vertices $a$ and $c$ are clearly symmetric, so we can consider removing $a$ and skip considering removing $c$. If $a$ is removed, swap $x$ and $y$ and note that making $x$ in the new position the source, both $b$ and $d$ become of type C, so there is no conflict with Theorem~\ref{relative-order} (the neighbourhoods in question are still consecutive).  
\item If $b$ is removed, or if $d$ is removed, then clearly there is no conflict with Theorem~\ref{relative-order}, and the neighbourhoods in question are still consecutive.
 \item The vertices $i$ and $j$ are clearly symmetric, so we can consider removing $i$ and skip considering removing $j$. The vertex $a$ becomes of degree 1 and can be also removed by Lemma~\ref{lemma-assumptions}. The obtained graph is word-representable by Theorem~\ref{main-2}.

\item The vertices $x$ and $z$ are clearly symmetric, so we can consider removing $x$ and skip considering removing $z$. The vertex $a$ becomes of degree 1 and can be also removed by Lemma~\ref{lemma-assumptions}. The obtained graph is word-representable by Theorem~\ref{main-2}.
\item If $y$ is removed, then the obtained graph is a subgraph of $K^{\triangle}_5$ and it is word-representable by Theorem~\ref{thm-K-Tri-m-w-r}. 
\end{itemize}
Our proof is completed. \end{proof}

\begin{theorem} The graph $T_9$ in Figure~\ref{nonRepT5-T9} is a minimal non-word-representable graph.\end{theorem}

\begin{figure}
\begin{center}
\begin{tikzpicture}[scale=1, transform shape]
 \tikzstyle{vertex}=[circle,fill=black,minimum size=3pt,inner sep=-1pt]
\path node[draw,shape=circle,minimum width=2cm,minimum height=2cm] (A) {$K_5$};
\foreach \vx/\lab in {45/j,135/i,210/x,270/y,330/z}
  \node[vertex] (v-\lab) at (A.\vx) {};
\foreach \vx in {1,...,5}
  \node[vertex]  (v-\vx) at (\vx*90:1.5cm) {} ;
   \node[vertex] (v-5) at (300:1.5cm) {};
 \foreach \from/\to in {2/i,2/x,1/i,1/j,4/j,4/z,3/i,3/j,3/y,3/x,5/z,5/j,5/i,5/y}
    { \draw (v-\from) -- (v-\to);}
\path (v-1) node[left] {{\small $b$}}
    (v-2) node[left] {{\small $a$}}
    (v-3) node[left] {{\small $d$}}
    (v-4) node[right] {{\small $c$}}
    (v-5) node[right] {{\small $e$}}
   (v-i) node[left] {{\small $i$}}
    (v-j) node[right] {{\small $j$}}
    (v-x) node[left] {{\small $x$}}
    (v-y) node[above] {{\small $y$}}
    (v-z) node[right] {{\small $z$}};
 \end{tikzpicture}
\caption{\label{T9-fig} Proving non-word-representability of $T_9$}
\end{center}
\end{figure}
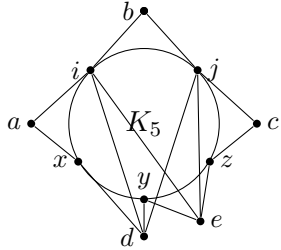

\begin{proof} We begin with proving non-word-representability of $T_9$. Suppose $T_9$ is word-representable, and thus, by Theorem~\ref{key-thm}, it can be oriented semi-transitively.  Pick any such semi-transitive orientation of $T_9$. Then, by Theorem~\ref{main-orientation}, the neighbourhoods of all vertices in the independent set must be consecutive intervals, and the only way to arrange this is shown in Figure~\ref{T9-fig} (where the five vertices of $\vec{P}$ are placed on a circle and the orientation of the longest path is assumed to be in clockwise direction). But then we obtain a contradiction with Theorem~\ref{relative-order} given by vertices $d$ and $e$. Indeed, 
\begin{itemize}
\item if $y$ is the source, or $j$ is the source, or $z$ is the source, then $d$ is of type A\&B and $e$ is of type C; the problem is then with $d$ being connected to $y$ and $i$.
\item if $x$ is the source, or $i$ is the source, then $e$ is of type A\&B and $d$ is of type C; the problem is then with $e$ being connected to $j$ and $y$.
\end{itemize}

For proving the minimality of $T_9$, we consider removing each of the vertices in $T_9$ (one at a time) and, if necessary, describe a permutation of vertices of $K_5$, which results in all neighbourhoods of vertices in $E_{n-5}$ be consecutive intervals on $\vec{P}$, or on whatever remains from $\vec{P}$ (which is still transitively oriented); then, Theorem~\ref{main-orientation} can be used to obtain a semi-transitive orientation of the resulting graph proving its word-representability by Theorem~\ref{key-thm}. 

\begin{itemize}
\item The vertices $a$ and $c$ are clearly symmetric, so we can consider removing $a$ and skip considering removing $c$. In the case of $a$ removed, swap $x$ and $y$ and note that making $x$ in the new position the source, both $d$ and $e$ become of type A\&B, so there is no conflict with Theorem~\ref{relative-order} (all neighbourhoods in question are still consecutive).  
\item If $d$ (resp., $e$) is removed, we can make $x$ (resp., $z$) the source and there will be no conflict with Theorem~\ref{relative-order}.
\item If $b$ is removed, then swapping $x$ and $i$, as well as $z$ and $j$, and making $z$ the source, we obtain both $d$ and $e$ being of type A\&B, so there is no conflict with Theorem~\ref{relative-order}.
\item The vertices $i$ and $j$ are clearly symmetric, so we can consider removing $i$ and skip considering removing $j$. The vertex $a$ becomes of degree 1 and can be also removed by Lemma~\ref{lemma-assumptions}. The obtained graph is word-representable by Theorem~\ref{main-2}. 
 \item The vertices $x$ and $z$ are clearly symmetric, so we can consider removing $x$ and skip considering removing $z$. The vertex $a$ becomes of degree 1 and can be also removed by Lemma~\ref{lemma-assumptions}. The obtained graph is word-representable by Theorem~\ref{main-2}. 
\item If $y$ is removed, then the obtained graph is word-representable by Theorem~\ref{main-2}. 
\end{itemize}
Our proof is completed. \end{proof}

\section{Word-representability of graphs obtained by gluing in a clique}\label{glueing-sec}

By {\em gluing} two graphs in a clique, we mean the following operation. Suppose $a_1,\ldots,a_k$ and $b_1,\ldots,b_k$ are cliques of size $k$ in graphs $G_1$ and $G_2$, respectively. Then, gluing $G_1$ and $G_2$ in a clique of size $k$ means identifying each $a_i$ with one $b_j$, for $i,j\in\{1,\ldots,k\}$ so that the neighbourhood of the obtained vertex $c_{i,j}$ is the union of the neighbourhoods of $a_i$ and $b_j$. 

By the hereditary nature of word-representability, if at least one of two graphs is non-word-representable, then gluing the graphs in a clique will result in a non-word-representable graph. Moreover, it is known that gluing two word-representable graphs in a vertex (a clique of size 1) always results in a word-representable graph (e.g. see \cite[Section 7.3]{K17} or \cite[Section 5.4.3]{KL15}). Further, it is not difficult to come up with examples when gluing two word-representable graphs in an arbitrary clique results in a word-representable graph; for a trivial such example, take two copies of a complete graph $K_n$, gluing which gives $K_n$, and $K_n$ can be represented by any permutation of length $n$. However, there are examples of word-representable graphs gluing which in an edge (a clique of size 2), or a triangle (a clique of size 3), results in a non-word-representable graph. The respective examples can be found in \cite[Section 5.4.3]{KL15}, and they are presented in Figures~\ref{ident-edge} and~\ref{gluing-in-triangle}, respectively. Thus, the rightmost graphs in these pictures are non-word-representable, while the other graphs are word-representable.

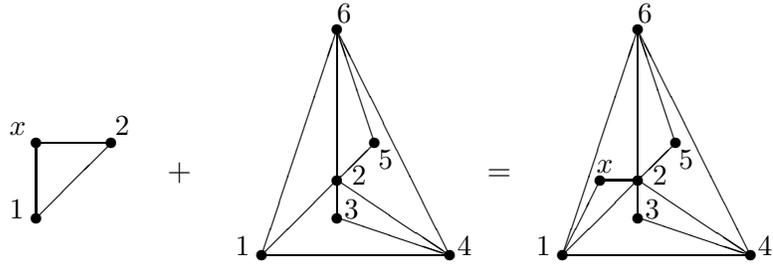
\begin{figure}
\begin{center}

\setlength{\unitlength}{5mm} 

\begin{picture}(17,4)

\put(0,1){

\put(0,0){\p} \put(0,2){\p}   \put(2,2){\p}

\put(-0.7,0){1} \put(-0.7,2.2){$x$}   \put(2.1,2.2){2} 

\put(0,0){\line(0,1){2}}
\put(0,2){\line(1,0){2}}
\put(0,0){\line(1,1){2}}

}

\put(3.5,2){
\put(0,0){{\large +}} 
}

\put(6,0){

\put(0,0){\p} \put(5,0){\p} \put(2,1){\p}  \put(2,2){\p} 
\put(2,6){\p} \put(3,3){\p} 

\put(-0.7,0){1} \put(5.2,0){4} \put(2.2,1){3}  \put(2.4,1.9){2} 
\put(2,6.2){6} \put(3.1,2.3){5}

\put(0,0){\line(1,0){5}}
\put(0,0){\line(1,1){3}}
\put(0,0){\line(1,3){2}}

\put(2,2){\line(0,-1){1}}
\put(2,2){\line(3,-2){3}}
\put(2,1){\line(3,-1){3}}

\put(2,2){\line(0,1){4}}
\put(5,0){\line(-1,2){3}}
\put(3,3){\line(-1,3){1}}

}

\put(12,2){
\put(0,0){{\large =}} 
}

\put(14,0){

\put(0,0){\p} \put(5,0){\p} \put(2,1){\p}  \put(2,2){\p} 
\put(2,6){\p} \put(3,3){\p} \put(1,2){\p} 

\put(-0.7,0){1} \put(0.9,2.2){$x$} \put(5.2,0){4} \put(2.2,1){3}  \put(2.4,1.9){2} 
\put(2,6.2){6} \put(3.1,2.3){5}

\put(0,0){\line(1,0){5}}
\put(0,0){\line(1,1){3}}
\put(0,0){\line(1,3){2}}
\put(0,0){\line(1,2){1}}

\put(1,2){\line(1,0){1}}

\put(2,2){\line(0,-1){1}}
\put(2,2){\line(3,-2){3}}
\put(2,1){\line(3,-1){3}}

\put(2,2){\line(0,1){4}}
\put(5,0){\line(-1,2){3}}
\put(3,3){\line(-1,3){1}}

}

\end{picture}
\caption{Gluing two word-representable graphs in an edge}
\label{ident-edge}
\end{center}
\end{figure}

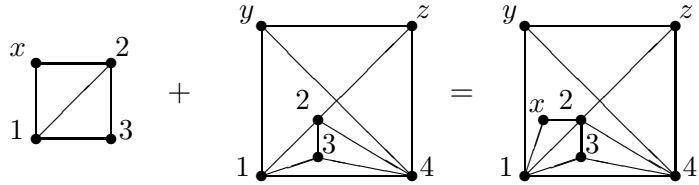
\begin{figure}
\begin{center}

\setlength{\unitlength}{5mm} 

\begin{picture}(17,4)

\put(0,1){

\put(0,0){\p} \put(0,2){\p} \put(2,0){\p}  \put(2,2){\p} 

\put(-0.7,0){1} \put(-0.7,2.2){$x$} \put(2.2,0){3}  \put(2.1,2.2){2} 

\put(0,0){\line(1,0){2}}
\put(0,0){\line(0,1){2}}
\put(2,0){\line(0,1){2}}
\put(0,2){\line(1,0){2}}
\put(0,0){\line(1,1){2}}

}

\put(3.5,2){
\put(0,0){{\large +}} 
}

\put(6,0){

\put(0,0){\p} \put(0,4){\p} \put(4,0){\p}  \put(4,4){\p} 
\put(1.5,1.5){\p} \put(1.5,0.5){\p} 

\put(-0.7,0){1} \put(-0.6,4.2){$y$} \put(4.2,0){4}  \put(4.1,4.2){$z$} 
\put(0.9,1.8){2} \put(1.6,0.6){3}

\put(0,0){\line(3,1){1.5}}
\put(4,0){\line(-5,1){2.5}}
\put(0,0){\line(1,0){4}}
\put(0,4){\line(1,0){4}}
\put(0,0){\line(1,1){4}}
\put(0,4){\line(1,-1){4}}
\put(0,0){\line(0,1){4}}
\put(4,0){\line(0,1){4}}
\put(1.5,1.5){\line(0,-1){1}}
\put(1.5,1.5){\line(5,-3){2.5}}

}

\put(11,2){
\put(0,0){{\large =}} 
}

\put(13,0){

\put(0,0){\p} \put(0,4){\p} \put(4,0){\p}  \put(4,4){\p} 
\put(1.5,1.5){\p} \put(1.5,0.5){\p} \put(0.5,1.5){\p} 

\put(-0.7,0){1} \put(-0.6,4.2){$y$} \put(4.2,0){4}  \put(4.1,4.2){$z$} 
\put(0.9,1.8){2} \put(1.6,0.6){3} \put(0.1,1.7){$x$}

\put(0,0){\line(1,3){0.5}}
\put(0,0){\line(3,1){1.5}}
\put(4,0){\line(-5,1){2.5}}
\put(0,0){\line(1,0){4}}
\put(0,4){\line(1,0){4}}
\put(0,0){\line(1,1){4}}
\put(0,4){\line(1,-1){4}}
\put(0,0){\line(0,1){4}}
\put(4,0){\line(0,1){4}}
\put(1.5,1.5){\line(0,-1){1}}
\put(1.5,1.5){\line(-1,0){1}}
\put(1.5,1.5){\line(5,-3){2.5}}

}

\end{picture}
\caption{Gluing two word-representable graphs in a triangle} \label{gluing-in-triangle}
\end{center}
\end{figure}

The question on whether gluing two word-representable graphs in a clique of size 4, or more, may result in a non-word-representable graph was open, though unpublished until \cite{K17}, for about ten years. In Subsection~\ref{constr-1} we use split graphs to show that gluing two word-representable graphs in a clique of size 4, or more, may result in a non-word-representable graph. A significance of our solution to the problem is in showing that gluing two cliques may be sensitive to which vertices are glued to which vertices, as the word-representability of the resulting graph may depend on it. In either case, in Subsection~\ref{constr-2}, we give another, surprisingly simple solution to the problem, which is based on a generalization of the construction in Figure~\ref{gluing-in-triangle}. 

\subsection{Solving the problem via split graphs}\label{constr-1}

Recall the definition of $K^{\triangle}_\ell$ in Section~\ref{split-gr-summary} (Definition~\ref{def-K-triang}) and the fact that $K^{\triangle}_\ell$ is word-representable by Theorem~\ref{thm-K-Tri-m-w-r}. Further, for $2\leq i\leq \ell$, let $K^i_\ell$ be the graph obtained from the complete graph $K_\ell$ labeled by $1,2,\ldots, \ell$, by adding a new vertex $x$ of degree 2 connected to the vertices $1$ and $i$. Clearly, any  $K^i_\ell$ is isomorphic to $K^2_\ell$, which is an induced subgraph of $K^{\triangle}_\ell$, and thus is word-representable.

Recall the definition of $A_\ell$ in Section~\ref{split-gr-summary} (Definition~\ref{def-Aell}) and the fact that $A_\ell$ is not word-representable by Theorem~\ref{Aell-min-non-repres}.

We observe that, for $\ell\geq 4$, gluing two word-representable graphs $K^{\triangle}_\ell$ and $K^i_\ell$, where  $2<i<\ell$, in the $\ell$-clique so that a vertex $j$ is glued with the vertex $j$ for $1\leq j\leq \ell$, results in a non-word-representable graph $G_i$. Indeed, $G_i$ contains the non-word-representable $A_{i}$ induced by the vertices $1,2,\ldots,(i+1),1',2',\ldots,(i-1)',x$. 

Note that even though $K^2_\ell$ (resp., $K^{\ell}_{\ell}$) is isomorphic to $K^i_\ell$ for  $2<i<\ell$, gluing the $\ell$-cliques in $K^{\triangle}_\ell$ and $K^2_\ell$ (resp., $K^{\ell}_{\ell}$) as above results in a word-representable graph $G_1$ (resp., $G_{\ell}$). Indeed, both of $G_1$ and $G_{\ell}$ are the graph $K^{\triangle}_\ell$ with the additional vertex $x$ having the same neighbourhood as  another vertex in $K^{\triangle}_\ell$ of degree 2. It is a direct corollary of Lemma~\ref{lemma-assumptions} that word-representability of $K^{\triangle}_\ell$ implies word-representability of $G_1$.  Thus, when glueing two word-representable graphs in a clique, the word-representability of the resulting graph may depend on how exactly we glue.

\subsection{Generalizing the known construction}\label{constr-2}

Here we present an alternative solution to the problem of gluing two graphs by generalizing the construction in Figure~\ref{gluing-in-triangle}.

Let $n\geq 2$ and $K'_n$ be the graph obtained from the complete graph $K_n$ on the vertex set $\{1,2,\ldots,n\}$ by adding a vertex $x$ connected to the vertices $1$ and $2$. For example, the leftmost graphs in Figures~\ref{ident-edge}, \ref{gluing-in-triangle}, \ref{gluing-in-K4} and~\ref{gluing-in-K5} are $K'_2$, $K'_3$, $K'_4$ and $K'_5$,  respectively. It is straightforward to check that the word  $x12x34\cdots n$ represents $K'_n$ for any $n\geq 2$. 

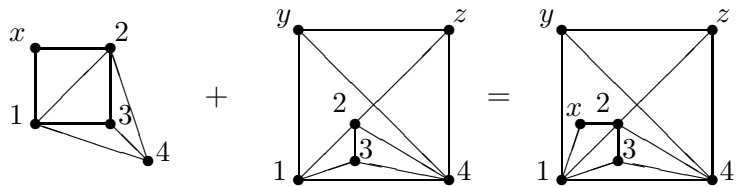
\begin{figure}[h]
\begin{center}

\setlength{\unitlength}{5mm} 

\begin{picture}(17,4)

\put(-1,1.5){

\put(0,0){\p} \put(0,2){\p} \put(2,0){\p}  \put(2,2){\p} \put(3,-1){\p} 

\put(-0.7,0){1} \put(-0.7,2.2){$x$} \put(2.2,0){3}  \put(3.2,-1){4} \put(2.1,2.2){2}  

\put(0,0){\line(1,0){2}}
\put(0,0){\line(0,1){2}}
\put(2,0){\line(0,1){2}}
\put(3,-1){\line(-1,1){1}}
\put(3,-1){\line(-1,3){1}}
\put(3,-1){\line(-3,1){3}}
\put(0,2){\line(1,0){2}}
\put(0,0){\line(1,1){2}}

}

\put(3.5,2){
\put(0,0){{\large +}} 
}

\put(6,0){

\put(0,0){\p} \put(0,4){\p} \put(4,0){\p}  \put(4,4){\p} 
\put(1.5,1.5){\p} \put(1.5,0.5){\p} 

\put(-0.7,0){1} \put(-0.6,4.2){$y$} \put(4.2,0){4}  \put(4.1,4.2){$z$} 
\put(0.9,1.8){2} \put(1.6,0.6){3}

\put(0,0){\line(3,1){1.5}}
\put(4,0){\line(-5,1){2.5}}
\put(0,0){\line(1,0){4}}
\put(0,4){\line(1,0){4}}
\put(0,0){\line(1,1){4}}
\put(0,4){\line(1,-1){4}}
\put(0,0){\line(0,1){4}}
\put(4,0){\line(0,1){4}}
\put(1.5,1.5){\line(0,-1){1}}
\put(1.5,1.5){\line(5,-3){2.5}}

}

\put(11,2){
\put(0,0){{\large =}} 
}

\put(13,0){

\put(0,0){\p} \put(0,4){\p} \put(4,0){\p}  \put(4,4){\p} 
\put(1.5,1.5){\p} \put(1.5,0.5){\p} \put(0.5,1.5){\p} 

\put(-0.7,0){1} \put(-0.6,4.2){$y$} \put(4.2,0){4}  \put(4.1,4.2){$z$} 
\put(0.9,1.8){2} \put(1.6,0.6){3} \put(0.1,1.7){$x$}

\put(0,0){\line(1,3){0.5}}
\put(0,0){\line(3,1){1.5}}
\put(4,0){\line(-5,1){2.5}}
\put(0,0){\line(1,0){4}}
\put(0,4){\line(1,0){4}}
\put(0,0){\line(1,1){4}}
\put(0,4){\line(1,-1){4}}
\put(0,0){\line(0,1){4}}
\put(4,0){\line(0,1){4}}
\put(1.5,1.5){\line(0,-1){1}}
\put(1.5,1.5){\line(-1,0){1}}
\put(1.5,1.5){\line(5,-3){2.5}}

}

\end{picture}
\caption{Gluing two word-representable graphs in $K_4$} \label{gluing-in-K4}
\end{center}
\end{figure}

Let the middle graph in Figure~\ref{gluing-in-triangle} be denoted by $M_4$, and for $n\geq 5$, $M_n$ is obtained by enlarging the clique formed by the vertices $1,2,3,4$ in $M_4$. That is, $M_n$ is obtained from $M_{n-1}$ by adding the vertex $n$ connected to all the vertices in $\{1,2,\ldots,n-1\}$ but not the vertices $y$ and $z$.  For example, $M_5$ is the middle graph in Figure~\ref{gluing-in-K5}. It is straightforward to check that the word $y1z4y2z3567\cdots n$ represents $M_n$ for any $n\geq 4$.\\

\begin{figure}[h]
\begin{center}

\setlength{\unitlength}{5mm} 

\begin{picture}(17,4)

\put(-1,1.5){

\put(0,-1){\p}  \put(0,0){\p} \put(0,2){\p} \put(2,0){\p}  \put(2,2){\p} \put(3,-1){\p} 

\put(-0.7,-1){5} \put(-0.7,0){1} \put(-0.7,2.2){$x$} \put(2.2,0){3}  \put(3.2,-1){4} \put(2.1,2.2){2}  

\put(0,-1){\line(0,1){1}}
\put(0,-1){\line(1,0){3}}
\put(0,-1){\line(2,1){2}}
\put(0,-1){\line(2,3){2}}
\put(0,0){\line(1,0){2}}
\put(0,0){\line(0,1){2}}
\put(2,0){\line(0,1){2}}
\put(3,-1){\line(-1,1){1}}
\put(3,-1){\line(-1,3){1}}
\put(3,-1){\line(-3,1){3}}
\put(0,2){\line(1,0){2}}
\put(0,0){\line(1,1){2}}

}

\put(3.5,2){
\put(0,0){{\large +}} 
}

\put(6,0.5){

\put(0,-1){\p} \put(0,0){\p} \put(0,4){\p} \put(4,0){\p}  \put(4,4){\p} 
\put(1.5,1.5){\p} \put(1.5,0.5){\p} 

\put(-0.7,-1){5}\put(-0.7,0){1} \put(-0.6,4.2){$y$} \put(4.2,0){4}  \put(4.1,4.2){$z$} 
\put(0.9,1.8){2} \put(1.6,0.6){3}

\put(0,-1){\line(0,1){1}}
\put(0,-1){\line(1,1){1.5}}
\put(0,-1){\line(4,1){4}}
\put(0,-1){\line(3,5){1.5}}
\put(0,0){\line(3,1){1.5}}
\put(4,0){\line(-5,1){2.5}}
\put(0,0){\line(1,0){4}}
\put(0,4){\line(1,0){4}}
\put(0,0){\line(1,1){4}}
\put(0,4){\line(1,-1){4}}
\put(0,0){\line(0,1){4}}
\put(4,0){\line(0,1){4}}
\put(1.5,1.5){\line(0,-1){1}}
\put(1.5,1.5){\line(5,-3){2.5}}

}

\put(11,2){
\put(0,0){{\large =}} 
}

\put(13,0.5){

\put(0,-1){\p} \put(0,0){\p} \put(0,4){\p} \put(4,0){\p}  \put(4,4){\p} 
\put(1.5,1.5){\p} \put(1.5,0.5){\p} \put(0.5,1.5){\p} 

\put(-0.7,-1){5}\put(-0.7,0){1} \put(-0.6,4.2){$y$} \put(4.2,0){4}  \put(4.1,4.2){$z$} 
\put(0.9,1.8){2} \put(1.6,0.6){3} \put(0.1,1.7){$x$}

\put(0,-1){\line(0,1){1}}
\put(0,-1){\line(1,1){1.5}}
\put(0,-1){\line(4,1){4}}
\put(0,-1){\line(3,5){1.5}}
\put(0,0){\line(1,3){0.5}}
\put(0,0){\line(3,1){1.5}}
\put(4,0){\line(-5,1){2.5}}
\put(0,0){\line(1,0){4}}
\put(0,4){\line(1,0){4}}
\put(0,0){\line(1,1){4}}
\put(0,4){\line(1,-1){4}}
\put(0,0){\line(0,1){4}}
\put(4,0){\line(0,1){4}}
\put(1.5,1.5){\line(0,-1){1}}
\put(1.5,1.5){\line(-1,0){1}}
\put(1.5,1.5){\line(5,-3){2.5}}

}

\end{picture}
\caption{Gluing two word-representable graphs in $K_5$} \label{gluing-in-K5}
\end{center}
\end{figure}
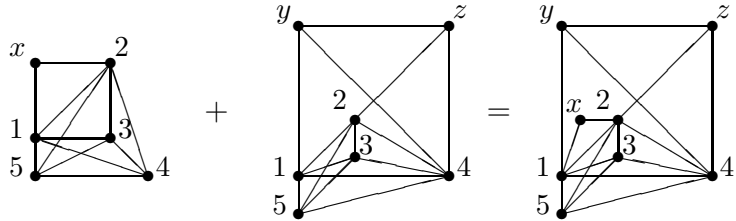

Finally, for $n\geq 4$, let $B_n$ be obtained from $M_n$ by adding a vertex $x$ connected just to the vertices $1$ and $2$. For example, $B_4$ is the rightmost graph in Figures~\ref{gluing-in-triangle} and~\ref{gluing-in-K4}, and $B_5$ is the rightmost graph in Figure~\ref{gluing-in-K5}. Note that using the hereditary nature of word-representable graphs, $B_n$ is not word-representable for any $n\geq 4$ since $B_4$ is not word-representable \cite{K17,KL15}.

Thus, for $n\geq 4$, gluing word-representable graphs $K'_n$ and $M_n$ in the clique formed by the vertices $1,2,\ldots,n$ gives the non-word-representable graph $B_n$, as desired. See Figures~\ref{gluing-in-K4} and~\ref{gluing-in-K5} for the cases of $n=4$ and $n=5$, respectively.

\section{Concluding remarks}

This paper extends our knowledge \cite{KLMW17} on word-representable split graphs, and the general theorems we prove, Theorems~\ref{degree-thm1} and~\ref{degree-thm2}, allow computational characterization of word-representable split graphs with cliques of size 5 in terms of 9 forbidden subgraphs. Taking into account that tackling the general case seems to be not feasible for the moment, a natural next step is in using our general theorems in (computational) characterization of word-representable split graphs with cliques of size 6, which we leave as an open research direction.

\section*{Acknowledgments}

The First author was partially supported by the National Natural Science Foundation of China (Grant Numbers\ 11901319) and
the Fundamental Research Funds for the Central Universities (Grant Number\ 63191349).

\end{document}